\documentclass{amsart}
\usepackage{amssymb} 
\newtheorem{thm}{Theorem}[section]

\newtheorem{lem}[thm]{Lemma}

\theoremstyle{definition}

\numberwithin{equation}{section}

\begin{document}
\title[]{Characterization of $\mathrm{SL}(2,q)$ \\ by its non-commuting graph}%
\author{Alireza Abdollahi}
\thanks{Author's Address: Department of Mathematics, University of Isfahan, Isfahan 81746-73441, Iran; and  School of Mathematics, Institute for Research in Fundamental Sciences (IPM), P.O.Box: 19395-5746, Tehran, Iran. e-mail: {\tt
abdollahi@member.ams.org}}
\thanks{This research was in part supported by a grant from IPM (No. 87200118)}%
\begin{abstract}
Let $G$ be a non-abelian group and $Z(G)$ be its center. The
non-commuting graph $\mathcal{A}_G$ of $G$ is the graph whose
vertex set is $G\backslash Z(G)$ and two vertices are joined by
an edge if they do not commute. Let  $\mathrm{SL}(2,q)$ be the
special linear group of degree 2 over the finite field of order
$q$. In this paper we prove that if $G$ is a group such that
$\mathcal{A}_G\cong \mathcal{A}_{\mathrm{SL}(2,q)}$ for some
prime power $q\geq 2$, then $G\cong \mathrm{SL}(2,q)$.
\end{abstract}
\maketitle
\section{\bf Introduction and Results}
Let $G$ be a non-abelian group and $Z(G)$ be its center. One can
associate with $G$ a graph whose vertex set is $G\backslash Z(G)$
 and two vertices are joined by an edge whenever they do not
commute. We call this graph the non-commuting graph of $G$ and it
will be denoted by $\mathcal{A}_G$. The non-commuting graph
$\mathcal{A}_G$ was first introduced by Paul Erd\"os \cite{N} to
formulate the following question: If every complete subgraph of
$\mathcal{A}_G$ is finite, is there a finite bound on the
cardinalities of complete subgraphs of $\mathcal{A}_G$? Neumann
\cite{N} answered positively Erd\"os question by proving that
$|G:Z(G)|=n$ is
finite and $n$ is obviously  the requested finite bound.\\
The non-commuting graph has been studied by many people (see e.g.,
\cite{AAM},\cite{MSZZ} and \cite{P}). It is proved in \cite{WS1}
(resp., in \cite{WS2}) that if $G$ is a finite group with
$\mathcal{A}_G\cong \mathcal{A}_{\mathrm{PSL}(2,q)}$ (resp.,
$\mathcal{A}_G\cong \mathcal{A}_{A_{10}}$), then $G\cong
\mathrm{PSL}(2,q)$ (resp., $G\cong A_{10}$). For any prime power
$q$, let $\mathrm{GL}(2,q)$ (resp. $\mathrm{SL}(2,q)$)  be the
general
 (resp. special)
linear group of degree 2 over the finite field of order $q$. In
this paper we study the groups  whose non-commuting graphs are
isomorphic to either $\mathrm{GL}(2,q)$ or $\mathrm{SL}(2,q)$.
Our main results are the following.
\begin{thm}\label{GL}
Let $G$ be a group such that $\mathcal{A}_G\cong
\mathcal{A}_{\mathrm{GL}(2,q)}$ for some prime power $q>3$. Then
$G/Z(G)\cong \mathrm{PGL}(2,q)$, $G'\cong\mathrm{SL}(2,q)$ and
$Z(G)$ is of order $q-1$. In particular,  if $q$ is even, then
$G=G' \times Z(G)$.
\end{thm}
\begin{thm}\label{SL}
Let $G$ be a group such that $\mathcal{A}_G\cong
\mathcal{A}_{\mathrm{SL}(2,q)}$ for some prime power $q\geq 2$.
Then $G\cong \mathrm{SL}(2,q)$.
\end{thm}
  For any prime power $q$, we
denote by $\mathrm{PGL}(2,q)$ (resp. $\mathrm{PSL}(2,q)$) the
projective general (resp. special) linear group of degree 2 over
the finite field of order $q$.
\section{\bf Proofs}
Here for convenience, we remind some of the properties of
non-commuting graphs and common properties of groups with
isomorphic non-commuting graphs.\\
 Let $G$ and $H$ be two
non-abelian groups such that $\mathcal{A}_G\cong \mathcal{A}_H$.
By Lemma 3.1 of \cite{AAM}, if one of $G$ or $H$ is finite, then
so is the other. The order of $\mathcal{A}_G$ is $|G|-|Z(G)|$ and
so $|G|-|Z(G)|=|H|-|Z(H)|$. The degree of a vertex $x$ in
$\mathcal{A}_G$ is equal to $|G|-|C_G(x)|$. Thus the multisets
 of  degrees of vertices of two graphs
$\mathcal{A}_G$ and $\mathcal{A}_H$ are the same.\\
A non-abelian group $G$ is called an $AC$-group, if the
centralizer  $C_G(x)$ of every non-central element $x$ of $G$ is
abelian. \\
Recall that a non-empty subset $X$ of the vertices of a simple
graph $\Gamma$ is called independent if every two distinct
vertices of $X$ are not joint by an edge in $\Gamma$. Thus an
independent set $S$ of  the non-commuting graph of a group is a
set of pairwise commuting non-central elements of the group.
\begin{lem}\label{lemm}
Let $G$ and $H$ be  two finite non-abelian group with
$\mathcal{A}_G\cong \mathcal{A}_H$. \\
{\rm (1)} \; If $|G|=|H|$, then the multisets (sets with
multiplicities) $\{|C_G(g)|\;:\; g\in G \backslash Z(G)\}$ and
$\{|C_H(h)| \;:\; h\in H \backslash
Z(H)\}$ are equal.\\
{\rm (2)} \; If $G$ is an $AC$-group, then $H$ is also an
$AC$-group.
\end{lem}
\begin{proof}
(1) \; It is straightforward, if we note that the set of
non-adjacent vertices to a vertex $x$   in the non-commuting graph
$H$ is $C_H(x)\backslash Z(H)$, and note that from $|G|=|H|$ we
also have $|Z(G)|=|Z(H)|$, since
$|H|-|Z(H)|=|G|-|Z(G)|$.\\
 (2) \; Note that a subgroup $S$ of a non-abelian group $K$
is abelian if and only if either $S\backslash Z(S)$ is empty or
$S\backslash Z(S)$ is an independent set in the non-commuting
graph $\mathcal{A}_K$.  Let $\phi$ be a graph isomorphism from
$\mathcal{A}_H$ onto $\mathcal{A}_G$. Then it is easy to see that
for each $h\in H\backslash Z(H)$,
$$C_H(h)\backslash Z(H)=\phi^{-1}\big(C_G(\phi(h))\backslash Z(G)\big).  \eqno{(*)}$$
Now since  $G$ is an $AC$-group,  $C_G(g)$ is abelian for all
$g\in G\backslash Z(G)$ and so it follows from $(*)$ and the
remark above that $C_H(h)$ is abelian. Hence $H$ is also an
$AC$-group.
\end{proof}
  Finite non-nilpotent $AC$-groups were completely
characterized by Schmidt \cite{S}.
 We use the following results in our proofs.
\begin{thm}{\rm (\cite[Satz 5.9.]{S})} \label{AC-nonsol}
Let $G$ be a finite non-solvable group. Then $G$ is an $AC$-group
if and only if $G$ satisfies one of the following conditions:
\begin{enumerate}
\item $G/Z(G)\cong \mathrm{PSL}(2,p^n)$ and $G'\cong
\mathrm{SL}(2,p^n)$, where $p$ is a prime and $p^n>3$.
\item  $G/Z(G)\cong\mathrm{PGL}(2,p^n)$ and $G'\cong
\mathrm{SL}(2,p^n)$, where $p$ is a prime and $p^n>3$.
\item $G/Z(G)\cong \mathrm{PSL}(2,9)$   and $G'$
 is a covering group of $A_6$. In particular, $G'$ is isomorphic to
\begin{align*} \mathcal{A}\cong<c_1,c_2,c_3,c_4, k \;|\;
c_1^3=c_2^2=c_3^2=c_4^2=(c_1c_2)^3=(c_1c_3)^2=&\\
=(c_2c_3)^3=(c_3c_4)^3=k^3, (c_1c_4)^2=k,&\\  c_2c_4=k^3c_4c_2,
kc_i=c_ik (i=1, \dots, 4), k^6=1>. &
\end{align*}
\item $G/Z(G)\cong \mathrm{PGL}(2,9)$   and $G'\cong\mathcal{A}$.
\end{enumerate}
\end{thm}
For a finite simple graph $\Gamma$, we denote by $\omega(\Gamma)$
the maximum size of a complete subgraph of $\Gamma$. So
$\omega(\mathcal{A}_G)$ is the maximum number of pairwise
non-commuting elements in a finite non-abelian group $G$.
\begin{thm}{\rm  (Satz 5.12 of \cite{S})}\label{AC-sol}
Let $G$ be a finite non-abelian solvable group.  Then $G$ is an
$AC$-group if and only if $G$ satisfies one of the following
properties:
\begin{enumerate}
\item $G$ is non-nilpotent and it has an abelian normal subgroup $N$ of prime index and
$\omega(\mathcal{A}_G)=|N:Z(G)|+1$.
\item $G/Z(G)$ is a Frobenius group with Frobenius kernel and
complement $F/Z(G)$ and $K/Z(G)$, respectively  and $F$ and $K$
are abelian subgroups of $G$; and
$\omega(\mathcal{A}_G)=|F:Z(G)|+1$.
\item  $G/Z(G)$ is a Frobenius group with Frobenius kernel and
complement $F/Z(G)$ and $K/Z(G)$, respectively; and  $K$ is an
abelian subgroup of $G$, $Z(F)=Z(G)$, and $F/Z(G)$ is of prime
power order; and
$\omega(\mathcal{A}_G)=|F:Z(G)|+\omega(\mathcal{A}_F)$.
\item $G/Z(G)\cong S_4$ and $V$ is a non-abelian subgroup of $G$
such that $V/Z(G)$ is the Klein $4$-group of $G/Z(G)$; and
$\omega(\mathcal{A}_G)=13$.
\item $G=A \times P$, where $A$ is an abelian subgroup and $P$ is an
$AC$-subgroup of prime power order.
\end{enumerate}
\end{thm}
\noindent{\bf Proof of Theorems \ref{GL} and \ref{SL}.} Let
$q_1=p_1^{n_1}>3$ and $q_2=p_2^{n_2}\geq 2$, where $p_1$ and $p_2$
are  two prime numbers. Let  $M_1=\mathrm{GL}(2,q_1)$ and
$M_2=\mathrm{SL}(2,q_2)$ and suppose that $G_1$ and $G_2$ are two
groups such that
$\mathcal{A}_{G_i}\cong \mathcal{A}_{M_i}$ for $i=1,2$.\\
If $q_1=2$, then $M_2\cong S_3$ is the symmetric group of degree
$3$ and so by Proposition 3.2 of \cite{AAM}, $G_2\cong M_2$. If
$q_2=3$, then $M_2$ is a group of order $24$ and its center has
order $2$.  As there is some
 element $g$ with $|C_{G_2}(g)|=6$, we see that there is no normal Sylow
 $3$-subgroup in
$G_2$. Hence $G_2/Z(G_2) \cong A_4$. So either $G_2 \cong M_2$ or
$\mathbb{Z}_2  \times A_4$. But as there are elements $h \in G_2$
with $|C_{G_2}(h)| = 4$, we have  $G_2 \cong M_2$.\\
Now let $q_2>3$. If $q_2$ is even, then $\mathrm{PSL}(2,q_2)\cong
M_2$ and so
$\mathcal{A}_{G_2}\cong\mathcal{A}_{\mathrm{PSL}(2,q_2)}$. Then by
Corollary 5.3 of \cite{AAM}, $G_2\cong \mathrm{PSL}(2,q_2)\cong
M_2$. Therefore we may assume that $q_2\geq 5$ is odd.\\
By Proposition 4.3 of \cite{AAM}, $|G_i|=|M_i|$  for $i=1,2$. By
Lemma 3.5 of \cite{AAM}, $M_i$'s are  $AC$-groups and so by Lemma
\ref{lemm}(2)   $G_i$'s are also $AC$-groups. Now since
$\mathcal{A}_{G_i}\cong \mathcal{A}_{M_i}$ and $|G_i|=|M_i|$, by
Lemma \ref{lemm} we have the following equality between  multisets
$$W_i=\{|C_{G_i}(x)| \;|\; x\in G_i\backslash Z(G_i)\}=\{|C_{M_i}(g)| \;|\; g\in
M_i\backslash Z(M_i)\}, \;\; i=1,2.$$ Also, since the order of two
graphs $\mathcal{A}_{G_i}$ and $\mathcal{A}_{M_i}$ are the same,
we have that $|G_i|-|Z(G_i)|=|M_i|-|Z(M_i)|$ and so
$|Z(G_i)|=|Z(M_i)|$ ($i=1,2$). Therefore, it follows from
Propositions 3.14 and 3.26 of \cite{AAM} that the multiset $W_1$
(resp., $W_2$) consists of three distinct integers $(q_1-1)^2$,
(resp., $(q_2-1)/2$) $q_1^2-1$ (resp., $(q_2+1)/2$) and
$q_1(q_1-1)$ (resp., $q_2$) with multiplicities
$\frac{q_i(q_i+1)}{2}$, $\frac{q_i(q_i-1)}{2}$
and $q_i+1$, respectively.\\

We claim that both groups $G_1$ and $G_2$ are not nilpotent.
Suppose, for a contradiction, that $G_i$ is nilpotent, then so is
$G_i/Z(G_i)$. Therefore $G_i/Z(G_i)$ has only one Sylow
$p_i$-subgroup. Since $W_1$ (resp., $W_2$) contains $q_i+1$
elements all equal to $q_1(q_1-1)$ (resp., $q_2$), there exist
two non-central elements $x_1$ and $y_1$ in $G_1$ (resp., $x_2$
and $y_2$ in $G_2$) such that $C_{G_1}(x_1)\not=C_{G_1}(y_1)$ and
$|C_{G_1}(x_1)|=|C_{G_1}(y_1)|=q_1(q_1-1)$ (resp.,
$C_{G_2}(x_2)\not=C_{G_2}(y_2)$ and
$|C_{G_2}(x_2)|=|C_{G_2}(y_2)|=2q_2$). Since $C_{G_i}(x_i)/Z(G_i)$
and $C_{G_i}(y_i)/Z(G_i)$ are of the same order $q_i$, they are
Sylow $p_i$-subgroups of $G_i/Z(G_i)$. It follows that
$C_{G_i}(x_i)/Z(G_i)=C_{G_i}(y_i)/Z(G_i)$ and so
$C_{G_i}(x_i)=C_{G_i}(y_i)$, a contradiction. \\

Now we prove that both $G_1$ and $G_2$ cannot be solvable.
Suppose, for a contradiction, that $G_i$'s are solvable. Then
since $G_i$ are not nilpotent, it follows from  Theorem
\ref{AC-sol} that $G_i$'s satisfy one of properties (1)-(4) in
Theorem \ref{AC-sol}. Since $q_i>3$ is a prime power and $q_2$ is
odd, both of $|G_1/Z(G_1)|=q_1(q_1^2-1)$ and
$|G_2/Z(G_2)|=\frac{q_2(q_2^2-1)}{2}$ cannot equal to $|S_4|=24$.
Therefore $G_i$'s do not satisfy $(4)$. If $G_i$ satisfies either
(1) or (2), then $W_i$ contains only two distinct elements, since
in the case (1), if $x\in N\backslash Z(G_i)$, then
$C_{G_i}(x)=N$; and if $x\in G\backslash N$ then $C_N(x)=Z(G_i)$;
so $|C_{G_i}(x)|\in\big\{|G_i:N||Z(G_i)|,|N|\big\}$ for every
non-central element $x\in G_i$, and in the case (3),
$|C_{G_i}(x)|\in\big\{|K|,|F|\big\}$. This is not possible, since
$W_i$ has exactly three distinct elements.\\
Finally, suppose that $G_i$ satisfies (3). Note that
$C_{G_i}(x)=C_F(x)$ for every non-central element $x\in F$ and
$C_{G_i}(x)$ is equal to the conjugate of $K$ which contains the
non-central element $x$. It follows that the three  distinct
elements of the multiset $W_i'=\{w/|Z(G)| \;|\; w\in W_i\}$ are
$|K/Z(G_i)|, r^k,r^\ell$, where $|F/Z(G)|=r^m$ and  $r$ is a prime
number. This is impossible, since no two of the numbers $q_1$,
$q_1+1$
or $q_1-1$ (resp., $q_2$, $(q_2+1)/2$ or $(q_2-1)/2$) can simultaneously be powers of the same prime.\\

Hence $G_i$'s are  finite non-solvable $AC$-groups. By Theorem
\ref{AC-nonsol}, $G_i$'s satisfy one of the conditions (1)-(4)
stated in Theorem \ref{AC-nonsol}.
 If $G_i$ satisfies (3), then  as $A_6$ has self-centralizing elements of order 4 and
5, $G_i$ contains two elements $x_i,y_i$ such that
$|\frac{C_{G_i}(x_i)}{Z(G_i)}|=4$ and
$|\frac{C_{G_i}(y_i)}{Z(G_i)}|=5$. This implies that
$q_1\in\{4,5\}$ and $q_2=9$. Therefore $|G_1/Z(G_1)|=4\cdot
(4^2-1)$ or $5\cdot (5^2-1)$, which is impossible, since
$|G_1/Z(G_1)|=|\mathrm{PSL}(2,9)|=\frac{9\cdot(9^2-1)}{2}$. Since
$M_2=\mathrm{SL}(2,9)$, $|Z(M_2)|=2$. But $3$ divides $Z(G_2)$ by
Theorem \ref{AC-nonsol}, a contradiction.\\
 If $G_i$ satisfies (4), then  as $\mathrm{PGL}(2,9)$ contains self-centralizing
elements of order 8 and 10, $G_i$ contains two elements $t_i$ and
$s_i$ such that $|\frac{C_{G_i}(t_i)}{Z(G_i)}|=8$ and
$|\frac{C_{G_i}(s_i)}{Z(G_i)}|=10$. It follows that
$\{8,10\}\subset \{q_2,\frac{q_2-1}{2},\frac{q_2-1}{2}\}$, which
is a contradiction as $q_2$ is a prime power; and for $i=1$,  it
follows that $q_1=9$. Hence $|Z(M_1)| = 8$. But 3 divides
$|Z(G_1)|$, a contradiction.
Thus $G_i$ does not satisfy both (3) and (4). \\
Now suppose that $G_i$ satisfies either (1) or (2). The group
$\mathrm{PGL}(2,r^m)$ (resp. $\mathrm{PSL}(2,r^m)$) has a
partition $\mathcal{P}$ consisting of $r^m+1$ Sylow $r$-subgroups,
$\frac{(r^m+1)r^m}{2}$ cyclic subgroups of order $r^m-1$ (resp.
$\frac{r^m-1}{\gcd(2,r^m-1)}$) and $\frac{(r^m-1)r^{m}}{2}$
cyclic subgroups of order $r^m+1$ (resp.
$\frac{r^m+1}{\gcd(2,r^m-1)}$) (see pp. 185-187 and p.193 of
\cite{H}). Now \cite[(5.3.3) in p.112]{S} states that if $x\in
G_i\backslash Z(G_i)$, then $C_{G_i}(x_i)/Z(G_i)$ belongs to
$\mathcal{P}$. Suppose that $G_i/Z(G_i)\cong \mathrm{PGL}(2,r^m)$
(resp. $\mathrm{PSL}(2,r^m)$). Thus there exist elements
$g_{i1},g_{i2},g_{i3}\in G_i\backslash Z(G_i)$ such that
$|C_{G_i}(g_{i1})|/|Z(G_i)|=r^m$,
$|C_{G_i}(g_{i2})|/|Z(G_i)|=r^m-1$ (resp.
$\frac{r^m-1}{\gcd(2,r^m-1)}$), $|C_{G_i}(g_{i3})|/|Z(G_i)|=r^m+1$
(resp. $\frac{r^m+1}{\gcd(2,r^m-1)}$).\\
 Therefore, if $G_i/Z(G_i)\cong \mathrm{PGL}(2,r^m)$ (resp. $\mathrm{PSL}(2,r^m)$),
 then $\{q_1-1,q_1,q_1+1\}=\{r^m-1,r^m,r^m+1\}$ (resp.
 $\{\frac{r^m-1}{\gcd(2,r^m-1)},r^m,\frac{r^m+1}{\gcd(2,r^m-1)}\}$ and $\{\frac{q_2-1}{2},\frac{q_2+1}{2},q_2\}=\{r^m-1,r^m,r^m+1\}$ (resp.
 $\{\frac{r^m-1}{\gcd(2,r^m-1)},r^m,\frac{r^m+1}{\gcd(2,r^m-1)}\}$).\\
 It follows that, if $G_2/Z(G_2)\cong \mathrm{PGL}(2,r^m)$
then $q_2=r^m+1$, $\frac{q_2+1}{2}=r^m$ and
$\frac{q_2-1}{2}=r^m-1$. Since $q_2\geq 5$, we have a
contradiction as $3\leq q_2-\frac{q_2-1}{2}=r^m+1-r^m+1=2$. Hence
$G_2/Z(G_2)\cong \mathrm{PSL}(2,r^m)$, $G_2'\cong
\mathrm{SL}(2,r^m)$ and $r^m=q_2$.  Now since
$|G_2'|=|G_2|=|M_2|$, we have that $G_2\cong
M_2=\mathrm{SL}(2,q_2)$. This completes the proof of Theorem
\ref{SL}.\\
Now if  $G_1/Z(G_1)\cong \mathrm{PGL}(2,r^m)$ (resp.
$\mathrm{PSL}(2,r^m)$), it follows that
 $q_1=r^m$ (resp. $q_1=2^m$). Since $\mathrm{PSL}(2,2^m)\cong
 \mathrm{PGL}(2,2^m)$, we have if $G_1$ satisfies either (1)  or
 (2), then $G_1/Z(G_1)\cong \mathrm{PGL}(2,q_1)$ and $G_1'\cong
 \mathrm{SL}(2,q_1)$.\\
 Therefore $G_1$ is a group satisfying the following conditions:
$$G_1/Z(G_1)\cong \mathrm{PGL}(2,q_1) \;\; (\bullet), \;\;  G_1'\cong
\mathrm{SL}(2,q_1)\;\;   \;\;\text{and}\;\; |Z(G_1)|=q_1-1.$$
 If $q_1=2^m$ for some integer $m>1$, then
$\mathrm{SL}(2,q_1)\cong \mathrm{PGL}(2,q_1)\cong
\mathrm{PSL}(2,q_1)$. Thus as $\mathrm{PSL}(2,q_1)$ is a
non-abelian simple group, it follows from $(\bullet)$ that
$G_1=G_1'Z(G_1)$; and since $G_1'$ is also non-abelian simple,
$G_1'\cap Z(G_1)=1$. Therefore $G_1=G_1' \times Z(G_1)$.
   This completes the proof of Theorem \ref{GL}. $\hfill \Box$\\

\noindent{\bf Acknowledgments.} This work was done during author's
sabbatical leave study  in Summer 2007 at ICTP, Trieste, Italy.
He is  grateful to University of Isfahan for its financial
support as well as  ICTP for their warm hospitality. He  was also
supported by the Center of Excellence for Mathematics, University
of Isfahan. The author is indebted to the referee for his/her
careful reading, valuable comments and pointing out a serious
error in the previous version of Theorem \ref{GL}.

\end{document}